\documentclass[a4paper]{article}

\usepackage{amssymb}
\usepackage{amsthm}
\usepackage{amsmath}

\newtheorem{thm}{Theorem}
\newtheorem{lem}{Lemma}

\newcommand*{\di}{\, d}
\newcommand*{\xin}{\ensuremath{1 + \xi_1 + \xi_1 \xi_2 + \cdots +
        \xi_1 \cdots \xi_n}}
\newcommand*{\xis}{\ensuremath{1 + \xi_1 + \xi_1 \xi_2 + \cdots}}
\newcommand*{\ev}{\ensuremath{\mathbf{E}}}
\newcommand*{\pr}{\ensuremath{\mathbf{P}}}
\newcommand*{\ivect}{\ensuremath{ \underline{i}}}
\newcommand*{\anum}[1]{\ensuremath{ a_k^{( #1 )}( \underline{i})}}
\newcommand*{\law}{\ensuremath{\overset{\mathsf{d}}{=}}}

\begin{document}

\title{An exponential functional of random walks}

\author{Tam\'as Szabados\footnote{Corresponding author, address:
Department of Mathematics, Budapest University of Technology and
Economics, M\H{u}egyetem rkp. 3, H \'ep. V em. Budapest, 1521,
Hungary, e-mail: szabados@math.bme.hu} \footnote{Research
supported by the French--Hungarian intergovernmental grant
``Balaton'' F-39/2000.} and Bal\'azs Sz\'ekely\footnote{Research
supported by the HSN laboratory of BUTE and by the Hungarian higher
education research program FKFP 0058/2001.} \\
Budapest University of Technology and Economics}

\date{}

\maketitle


\begin{abstract}

The aim of this paper is to investigate discrete approximations of the
exponential functional $\int_0^{\infty} \exp(B(t) - \nu t) \di t$ of
Brownian motion (which plays an important role in Asian options of
financial mathematics) by the help of simple, symmetric random
walks. In some applications the discrete model could be even more
natural than the continuous one. The properties of the discrete
exponential functional are rather different from the continuous one:
typically its distribution is singular w.r.t.  Lebesgue measure, all
of its positive integer moments are finite and they characterize the
distribution. On the other hand, using suitable random walk
approximations to Brownian motion, the resulting discrete exponential
functionals converge a.s. to the exponential functional of Brownian
motion, hence their limit distribution is the same as in the
continuous case, namely, the one of the reciprocal of a gamma random
variable, so absolutely continuous w.r.t. Lebesgue measure. This way
we give a new, elementary proof for an earlier result by Dufresne and
Yor as well.

\end{abstract}


\renewcommand{\thefootnote}{\alph{footnote}}
\footnotetext{ 2000 \emph{MSC.} Primary
 60F17, 60F25. Secondary 62P05, 60G50.}

\section{Introduction}

The geometric Brownian motion (originally introduced by the economist
P. Samuelson in 1965) plays a fundamental role in the Black--Scholes
theory of option pricing, modeling the price process of a stock. It
can be explicitly given in terms of Brownian motion (BM) $B$ as
\[
S(t) = S_0 \exp\left(\sigma B(t) + \left(\mu - \sigma^2/2 \right)t
\right), \qquad t \ge 0.
\]
In the case of Asian options one is interested in the average price
process
\[
A(t) = \frac{1}{t} \int_0^t S(u) \di u , \qquad t \ge 0.
\]
The following interesting result is true for the distribution of a
closely related, widely investigated exponential functional of BM:
\begin{equation} \label{eq:I}
\mathcal{I} = \int_0^{\infty} \exp(B(t) - \nu t) \di t \law
\frac{2}{Z_{2\nu}} \qquad (\nu > 0).
\end{equation}
Here $Z_{2\nu}$ is a gamma distributed random variable with index
$2\nu$ and parameter $1$, while $\law$ denotes equality in distribution.
This result was proved by \cite{Dufresne (1990)} using discrete
approximations with gamma distributed random variables and also by
\cite{Yor (1992)}, using rather ingenious stochastic analysis tools.
For more background information see \cite{Yor (2001)} and
\cite{Csorgo (1999)}.

As a consequence, the $p$th integer moment of $\mathcal{I}$ is finite
iff $p < 2\nu$ and
\begin{equation} \label{eq:Ip}
\ev(\mathcal{I}^p) = 2^p \frac{\Gamma(2\nu -p)}{\Gamma(2\nu)}.
\end{equation}
On the other hand, all negative integer moments, also given by
(\ref{eq:Ip}), are finite and they characterize the distribution of
$\mathcal{I}$.

The situation is much nicer when BM with negative drift is replaced
in the model by the negative of a subordinator $(\alpha_t, t\ge 0)$,
that is, by the negative of a non-decreasing process with stationary
and independent increments, starting from the origin. Then, as was
shown by \cite{CarPetYor (1997)}, all positive integer moments of
$\mathcal{J} = \int_0^{\infty} \exp(-\alpha_t) \di t$ are finite:
\begin{equation} \label{eq:Jp}
\ev(\mathcal{J}^p) = \frac{p!}{\Phi(1)\cdots \Phi(p)}, \qquad
\Phi(\lambda) = -\frac{1}{t} \log \ev(\exp(-\lambda \alpha_t)),
\end{equation}
and in this case the positive integer moments characterize the
distribution of $\mathcal{J}$.

To achieve a similar favorable situation in the BM case, at least in
an approximate sense, it is a natural idea to use a simple, symmetric
random walk (RW) as an approximation, with a large enough negative
drift. Besides, in some applications a discrete model could be more
natural than a continuous one. It seems important that, as we shall
see below, the discrete case is rather different from the continuous
case in many respects.

So let $(X_j)_{j=1}^{\infty}$ be an i.i.d. sequence with $\pr(X_1 =
\pm 1) = \frac12$ and $S_0=0$, $S_k = \sum_{j=1}^k X_j$ $(k \ge
1)$. Introduce the following approximation of $\mathcal{I}$:
\begin{equation}\label{eq:Y}
Y = \sum_{k=0}^{\infty} \exp(S_k - k \nu ) = \xis , \qquad
\xi_j = \exp(X_j - \nu) ,
\end{equation}
where $\nu > 0$. In this paper we investigate the properties of $Y$,
which will be called the discrete exponential functional of the
given RW, or shortly, the discrete exponential functional.

In Section 2 below it turns out that the distribution of $Y$ is
singular w.r.t. Lebesgue measure if $\nu >1$. Then in Section 3 we
find a formula, similar to (\ref{eq:Ip}), for all positive integer
moments of $Y$ when $\nu >1$, and, because $Y$ is bounded then, these
moments really characterize its distribution. Finally, in Section 4
we use a nested sequence of RWs to obtain a.s.
converging approximations of $\mathcal{I}$, and this way an
elementary proof of result (\ref{eq:I}) of Dufresne and Yor as well.

\section{The distribution of the discrete exponential functional}

Let us start with a natural generalization: $(\xi_j)^{\infty}_{j=1}$
be i.i.d., $\xi_j > 0$. Consider first the finite polynomial
\begin{eqnarray} \label{eq:Yn}
Y_n &=& \xin \\
&=& 1 + \xi_1(1 + \xi_2 + \xi_2 \xi_3 + \cdots + \xi_2 \cdots \xi_n)
\qquad (n \ge 1), \nonumber
\end{eqnarray}
$Y_0 = 1$. This implies the following equality in distribution:
\begin{equation} \label{eq:law}
Y_n \law 1 + \xi Y_{n-1},
\end{equation}
where $\xi \law \xi_1$, and $\xi$ is independent of $Y_{n-1}$. Since
$Y_n \nearrow Y = \xis$ a.s., we get the basic \emph{self-similarity} of
$Y$ in distribution:
\begin{equation} \label{eq:self}
Y \law 1 + \xi Y,
\end{equation}
where $\xi$ is independent of $Y$. We remark that infinite polynomials
similar to $Y$ were studied by \cite{Vervaat (1979)} and many others.
There some of the ideas discussed below have already appeared.

A standard application of the strong law of large numbers gives a
condition for having an a.s. finite limit $Y$ here, see Theorem 1 in
\cite{Szekely (1975)}. Namely, when $\ev(|\log \xi_j|) < \infty$,
one has $Y_n \nearrow Y < \infty$ a.s. if and only if $\ev(\log \xi_j) < 0$.

In the special case when $Y$ is defined as in (\ref{eq:Y}), but $S_n$
is the partial sum of an \emph{arbitrary} i.i.d. sequence
$(X_j)_{j=1}^{\infty}$ with zero expectation, $Y < \infty$ a.s. iff
the drift added is negative: $\nu > 0$. Hence this condition is
always assumed in our basic example (simple, symmetric RW).

Next we want to show that self-similarity (\ref{eq:self}) implies a
simple functional equation for the distribution function $F(y)=\pr(Y
\le y)$, $y \in \mathbb{R}$. For a modest generalization of our basic
case, let us introduce some notations. In (\ref{eq:Yn}) let $\xi_j$
take the positive values $\gamma_1 < \dots < \gamma_N$, and let
$p_i=\pr (\xi = \gamma_i)$. (In our basic case $N=2$, $\gamma_1 =
e^{-1-\nu}$, $\gamma_2 = e^{1-\nu}$, $p_1 = p_2 = \frac12$.) Consider
the following similarity transformations: $T_i(x) = \gamma_i x +1$
$(1 \le i \le N)$. When
\[
\ev(\log \xi) = \sum^{N}_{i=1} p_i \: \log \gamma_i < 0
\]
holds, by (\ref{eq:self}) we have $\pr(Y \le y) = \pr(1 + \xi Y \le y)
= \sum^N_{i=1} p_i \: \pr(1 + \gamma_i Y \le y | \xi = \gamma_i)
= \sum^N_{i=1} p_i \: \pr(1 + \gamma_i Y \le y) $. Thus one obtains
the following functional equation for the distribution function of $Y$:
\begin{equation} \label{eq:df}
F(y)= \sum^N_{i=1} p_i \: F(T^{-1}_{i}(y)).
\end{equation}

An important special case is when $\gamma_N < 1$ (in the basic case:
$\nu >1$). Then by (\ref{eq:Yn}), $Y$ is a bounded random variable.
Moreover, each mapping $T_i$ is a contraction, having a unique
fixpoint $y_i = (1 - \gamma_i)^{-1}$, $0 < y_1 < \cdots < y_N <
\infty$. Since each $T_i$ is an increasing function, $T_i(y_j) <
T_i(y_k)$ if $j < k$. Also, $T_j(y_i) < T_k(y_i)$ if $j < k$. Then it
follows that each $T_i$ maps the fundamental interval $I = [y_1, y_N]$
into itself. Clearly, $I$ contains the range of $Y$ too.

In the case $\gamma_N < 1$ it is useful to rephrase the given problem
in the language of fractal theory, see e.g. \cite{Falconer (1990)}.
Let us introduce the symbolic space $\Sigma = \{\ivect = (i_1, i_2,
\dots) : i_j = 1 , \dots , N \}$, endowed with the countable power of
the discrete measure $(p_1, \dots , p_N)$, denoted by $\pr$. By
(\ref{eq:Yn}),
\begin{equation} \label{eq:Yn2}
Y_n = 1+\gamma_{i_1}(1+\gamma_{i_2}(\cdots (1+\gamma_{i_n})))
= (T_{i_1} \circ T_{i_2} \circ \cdots \circ T_{i_n})(1),
\end{equation}
with probability $p_{i_1} p_{i_2} \cdots p_{i_n}$. Thus the canonical
projection $\Pi: \Sigma \to I$, $\Pi(\ivect) = \lim_{k \to \infty}
(T_{i_1} \circ \cdots \circ T_{i_k})(1) = \lim_{k \to \infty }
(1+\gamma_{i_1} + \cdots + \gamma_{i_1} \dots \gamma_{i_k})$ maps
$\Sigma$ onto the range of $Y$. The
\emph{attractor} $\Lambda$ of the iterated function scheme of
similarity transformations $(T_1, \dots , T_N)$ is defined as
\begin{equation} \label{eq:Lambda}
\Lambda = \bigcap_{k \ge 0} \bigcup_{1 \le i_1, \dots, i_k \le N}
(T_{i_1} \circ \cdots \circ T_{i_k})(I),
\qquad \Lambda = \bigcup_{i=1}^N T_i(\Lambda).
\end{equation}
Then $\Lambda$ is a non-empty, compact, self-similar set. In
(\ref{eq:Lambda}) the fundamental interval $I = [y_1, y_N]$ can be
replaced by any interval $J$ which is mapped into itself by each
$T_i$, e.g. by $J = [0, y_N] \ni 1$. Thus range$(Y) \subset \Lambda$,
cf. (\ref{eq:Yn2}). The converse is also true, since for any $y \in
\Lambda$ and for any $\epsilon > 0$ there is an $\ivect \in
\Sigma$ and a large enough $k$ such that $y \in (T_{i_1} \circ \cdots
\circ T_{i_k})(J)$ and the length $|(T_{i_1} \circ \cdots \circ
T_{i_k})(J)| = \gamma_{i_1} \cdots \gamma_{i_k} |J| < \epsilon$.
Hence, by (\ref{eq:Yn2}), $y \in $ range$(Y)$, that is, $\Lambda =
$range$(Y)$. Also, the distribution of $Y$ on the real line, which
will be denoted by $\pr_Y$, is simply $\pr \circ \Pi^{-1}$.

We are going to use the notations
\begin{equation} \label{eq:Ii1k}
I_{i_1 \dots i_k} = [y_{i_1 \dots i_k 1}, y_{i_1 \dots i_k N}] =
(T_{i_1} \circ \cdots \circ T_{i_k})(I),
\end{equation}
$y_{i_1 \dots i_k l} = (T_{i_1} \circ \cdots \circ T_{i_k})(y_l)$ $(l
= 1, \dots, N)$ as well, where $i_j = 1, \dots N$ and $y_l =
(1-\gamma_l)^{-1}$. The length of such an interval is $|I_{i_1 \dots
i_k}| = \gamma_{i_1} \cdots \gamma_{i_k} |I|$, where $|I| = y_N -
y_1$.

Returning to the distribution of $Y$ in the basic case, consider
first when the intervals $I_1 = T_1(I) = [y_{11}, y_{12}] = [y_1,
y_{12}]$ and $I_2 = T_2(I) = [y_{21}, y_{22}] = [y_{21}, y_2]$ do not
overlap, where $y_{12} = 1+\gamma_1(1-\gamma_2)^{-1}$, $y_{21} =
1+\gamma_2(1-\gamma_1)^{-1}$. Thus there is no overlap iff $y_{12} <
y_{21}$, i.e., $\nu > \log(e+e^{-1}) \approx 1.127$. Since $F(y)=0$
if $y < y_1$ and $F(y)=1$ if $y \ge y_2$, in this non-overlapping
case (\ref{eq:df}) simplifies as
\begin{equation}
F(y) = \left\{
\begin{array}{ll}
 \frac12 F(T^{-1}_{1}(y)) & \mbox{if } y \in [y_1, y_{12}), \\
 \frac12 & \mbox{if } y \in [y_{12}, y_{21}), \\
 \frac12 + \frac12 F(T^{-1}_{2}(y)) & \mbox{if } y \in [y_{21},
 y_2).
\end{array}
\right. \label{eq:F}
\end{equation}

By the similarities given by $T_1$ and $T_2$, applied to (\ref{eq:F}),
one obtains that $F$ has constant value $\frac14$ over the interval
$[y_{112}, y_{121})$ and constant value $\frac34$ on $[y_{212},
y_{221})$. Continuing this way by induction one gets that $F$ has
constant dyadic values over such \emph{plateau} intervals:
\begin{equation} \label{eq:plat}
F(y) = 2^{-k-1} + \sum^{k}_{j=1} (i_j - 1) 2^{-j}, \qquad y \in [y_{i_1
\dots i_k 12}, y_{i_1 \dots i_k 21}), \qquad i_j = 1, 2.
\end{equation}

The sum of the lengths of these plateaus is $|I| \left(1 - (\gamma_1 +
\gamma_2)\right) \left(1 + (\gamma_1+\gamma_2) \right.$ $\left.  +
(\gamma_1+\gamma_2)^2 \right.$ $\left.+ \cdots \right)$, so add up to
$|I|$.  Hence the attractor $\Lambda$ (the range of $Y$), i.e., the
set of points of increase of $F$, has zero Lebesgue measure. So it is
a Cantor-type set: an uncountable, perfect set of Lebesgue measure
zero.

The distribution function $F$ is clearly a continuous singular
function. For, if $y_0 \in \Lambda$ and $\epsilon > 0$ is given, take
$k$ so that $2^{-k} < \epsilon$. By the construction of $\Lambda$,
there exists an interval $I_{i_1 \dots i_k} \ni y_0$. Let the left
endpoint of the left neighbor plateau of $I_{i_1 \dots i_k}$ be
$\eta_1$ (or $-\infty$), and the right endpoint of the right neighbor
plateau be $\eta_2$ (or $\infty$). If $\delta = \min(y_0 - \eta_1,
\eta_2 - y_0) > 0$, then for any $y$ such that $|y - y_0| < \delta$
one has $|F(y) - F(y_0)| \le 2^{-k} < \epsilon$ by (\ref{eq:plat}).

It is not difficult to see, cf. \cite{Grincevicius (1974)}, that in
general, any solution of (\ref{eq:self}) has either absolutely
continuous or continuous singular distribution.

We mention that standard results of fractal theory, see Theorem 9.3 in
\cite{Falconer (1990)}, imply that the Hausdorff dimension $s$ of
$\Lambda$ equals the (unique) solution of the equation $\gamma_1^s +
\gamma_2^s = 1$. Solving this equation for $\nu$, we get $\nu = s^{-1}
\log(e^s + e^{-s})$. Hence the fractal dimension $s$ is a strictly
decreasing function of $\nu > \log(e + e^{-1})$, tending to $1$ as
$\nu \to \log(e + e^{-1})$ and converging to $0$ as $\nu \to \infty$.
Also, the Hausdorff measure of $\Lambda$ is $\mathcal{H}^s(\Lambda) =
|I|^s$, where the Hausdorff dimension $s$ is the one defined above.
It means that
\[
\mathcal{H}^s(\Lambda) = \left( \left(1 - e (2\cosh s)^{-1/s}
\right)^{-1} - \left(1 - e^{-1} (2\cosh s)^{-1/s} \right)^{-1}
\right)^s.
\]
Thus $\mathcal{H}^s(\Lambda) \to e^2 - e^{-2}$ as $\nu \to \log(e +
e^{-1})$ and $\mathcal{H}^s(\Lambda) \to 0$ as $\nu \to \infty$.

Next we are going to show that the distribution of $Y$ is singular
w.r.t. Lebesgue measure even in the overlapping case if $\nu > 1$.
Again, we consider the slight generalization introduced above. The
proof below is based on \cite{SiSoUr (2001)} and on personal
communication with K. Simon.

\begin{thm} \label{th:singular}
Let $\xi$ take the values $\gamma_i$ $(i = 1, \dots, N)$, $0 <
\gamma_1 < \dots < \gamma_N < 1$, and let $p_i=\pr (\xi =
\gamma_i)$. Take an i.i.d. sequence $(\xi_j)^{\infty}_{j=1}$, $\xi_j
\law \xi$.  Then the distribution of $Y=\xis$ is singular
w.r.t. Lebesgue measure, if
\[
-\chi_{\pr} = \ev(\log \xi) = \sum^{N}_{i=1} p_i \log \gamma_i
< \sum^{N}_{i=1} p_i \log p_i = -H_{\pr}.
\]
This will be called the entropy condition. Here $\chi_{\pr}$ is the
Lyapunov exponent of the iterated function scheme $(T_1, \dots, T_N)$
corresponding to the Bernoulli measure $\pr$.
\end{thm}

\begin{proof}
We are going to use the fractal theoretical approach and notations
introduced above.

We want to show that
\begin{equation} \label{eq:DPY}
(\bar{D}\pr_Y)(x) = \limsup_{r \searrow 0}
\frac{\pr_Y(B(x,r))}{\lambda(B(x,r))} = \infty \qquad \pr_Y \mbox{a.s.},
\end{equation}
where $B(x,r)$ denotes the open ball (in the real line) with center
at $x$ and radius $r$ and $\lambda$ is Lebesgue measure. The
statement of the theorem easily follows from this. For, take the set
$E = \{x \in I : (\bar{D}\pr_Y)(x) = \infty \}$. Then (\ref{eq:DPY})
implies that $\pr_Y(E) = 1$, while e.g. Theorem 8.6 in \cite{Rudin
(1970)} shows that the symmetric derivative $D\pr_Y$ exists and is
finite $\lambda$ a.e., so $\lambda(E) = 0$.

Introduce the notation $\anum{j} = \# \{l : i_l=j, 1 \le l \le k \}$.
Thus
\[
\Pi(\ivect) = 1 + \sum_{k=1}^{\infty} \prod^N_{j=1} \gamma_j^{\anum{j}}.
\]
By the SLLN, the set $A_j = \left\{ \ivect \in \Sigma : k^{-1} \anum{j}
\to p_j \right\}$ has probability $1$ for every $j = 1, \dots , N$ and
so has $A = \bigcap_{j=1}^N A_j$. Let $C = \{x \in I : \Pi^{-1}(x) \cap
A \ne \emptyset \}$. Then $\pr_Y(C)=1$.

If $x \in C$, there exists $\ivect \in A$ such that $\Pi(\ivect) = x$
and $k^{-1} \anum{j} \to p_j$ as $k \to \infty$ for all $j = 1, \dots
, N$. Fix such an $\ivect$ and $x$. Let $r_k$ be the smallest radius
such that $B(x,r_k) \supset I_{i_1 \dots i_k}$, where $ \ivect = (i_1,
\dots, i_k, \dots)$ and $I_{i_1 \dots i_k}$ is defined by
(\ref{eq:Ii1k}).

The following facts are clear: (a) $x \in \Lambda$, moreover, $x \in
I_{i_1 \dots i_k}$, see (\ref{eq:Lambda}) and (\ref{eq:Ii1k}); (b)
$\frac12 |I_{i_1 \dots i_k}| < r_k \leq c |I_{i_1\dots i_k}|$, where
$c > 1$ is arbitrary; (c) $|I_{i_1 \dots i_k}| = |I| \prod^N_{j=1}
\gamma_j^{\anum{j}}$; (d) $\pr_Y(B(x,r_k)) \ge \pr_Y(I_{i_1 \dots i_k})
= \pr(i_1, \dots, i_k) = \prod^N_{j=1} p_j^{\anum{j}}$. Using these
facts it follows for any $k \ge 1$ that
\begin{eqnarray*}
\frac{\pr_Y(B(x,r_k))}{\lambda(B(x,r_k))} \ge (2 c |I|)^{-1} \left(
\frac{\prod^N_{j=1} p_j^{k^{-1} \anum{j}}} {\prod^N_{j=1}
\gamma_j^{k^{-1} \anum{j}}} \right)^k .
\end{eqnarray*}

By our assumptions concerning $x$ and $\ivect$, the ratio on the right
hand side converges to $\left(p_1^{p_1} \cdots p_N^{p_N}\right) /
\left(\gamma_1^{p_1} \cdots \gamma_N^{p_N}\right)$ as $k \to
\infty$. The entropy condition of the theorem implies that this latter
ratio is larger than $1$. Hence (\ref{eq:DPY}) holds, and this
completes the proof.
\end{proof}

Returning to our basic case, consider the entropy condition when
$\gamma_1=e^{-1-\nu}$, $\gamma_2=e^{1-\nu}$ and
$p_1=p_2=\frac{1}{2}$. The condition holds iff $\nu > \log 2 \approx
0.693$, since this is equivalent to $\frac12 (-1-\nu) + \frac12
(1-\nu) < \frac12 \log \frac12 + \frac12 \log \frac12$. Combining
this with the condition $\gamma_2 < 1$, this means that the
distribution of $Y$ is singular w.r.t. Lebesgue measure for any $\nu
> 1$.

Characterization of the distribution of $Y$ when $0 < \nu \le 1$
remains open. In that case one of the two similarity mappings, $T_2$,
is not a contraction anymore, and that situation requires more
sophisticated tools than the ones above.

\section{The moments of the discrete exponential functional}

Let us consider first the general case: $(\xi_j)^{\infty}_{j=1}$
i.i.d., $\xi_j > 0$, as at the beginning of the previous section. Now
we turn our attention to the moments of $Y = \xis$. If $Y_n$ is
defined by (\ref{eq:Yn}), the equality in law (\ref{eq:law}) implies
\begin{equation} \label{eq:EYn}
\ev(Y_n^p) = \ev\left((1 + \xi Y_{n-1})^p\right) = \sum^p_{k=0}
\binom{p}{k} \mu_k \ev(Y_{n-1}^k),
\end{equation}
where $p \ge 0$ integer and $\mu_k = \ev(\xi^k)$. As $n \to \infty$,
by monotone convergence we obtain
\begin{equation} \label{eq:EY}
\ev(Y^p) = \sum^p_{k=0} \binom{p}{k} \mu_k \ev(Y^k).
\end{equation}
In (\ref{eq:EYn}) and (\ref{eq:EY}) both sides are either finite
positive, or $+\infty$.

\begin{thm} \label{th:moments}
Let $(\xi_j)^{\infty}_{j=1}$ be an i.i.d. sequence, $\xi_j > 0$, and
$Y = \xis$. For $p \ge 1$ real, $\ev(Y^p) < \infty$ if and only if
$\mu_p = \ev(\xi^p) < 1$. Then $\mu_q < 1$ for any $1 \le q \le p$ and
$\ev(Y^p) \le (1 - \mu_p^{1/p})^{-p}$ as well. In this case if $p \ge 1$
is an integer, we also have the recursion formula
\begin{equation}\label{eq:moments}
\ev(Y^p) = \frac{1}{1-\mu_p} \sum_{k=0}^{p-1} \binom{p}{k} \mu_k \ev(Y^k).
\end{equation}
\end{thm}

\begin{proof}
First, (\ref{eq:law}) and the simple inequality $(1 + x)^p \ge 1 +
x^p$ ($x \ge 0$, $p \ge 1$, real) imply that $\ev(Y_n^p)
\ge 1 + \mu_p \ev(Y_{n-1}^p)$. Suppose that $\mu_p \ge 1$. Since
$\ev(Y_0^p) = 1$, taking limit as $n \to \infty$, one gets that
$\ev(Y^p) = \infty$.

Conversely, suppose that $\mu_p < 1$. Then by H\"older's (or by
Jensen's) inequality, $\mu_q \le \mu_p^{q/p} < 1$ for any $1 \le q \le
p$ as well. We want to show that $\ev(Y^p)$ is finite. Let us begin
by observing that $\ev\left((Y_j - Y_{j-1})^p \right) =
\ev\left((\xi_1 \cdots \xi_j)^p \right) = \mu_p^j$ ($j \ge 1$). Hence
by the triangle inequality and $Y_0 = 1$ we get that
\[
\ev(Y_n^p)^{1/p}  \le 1 + \sum_{j=1}^n \ev\left((Y_j - Y_{j-1})^p
\right)^{1/p} = \sum_{j=0}^n \mu_p^{j/p} < \frac{1}{1 - \mu_p^{1/p}} ,
\]
for any $n \ge 1$ when $\mu_p < 1$. Taking limit as $n \to \infty$,
it follows that $\ev(Y^p) \le (1 - \mu_p^{1/p})^{-p} < \infty$.

Finally, the recursion formula (\ref{eq:moments}) directly follows
from (\ref{eq:EY}) when $\mu_p < 1$, $p \ge 1$ integer.
\end{proof}

For integer $p \ge 1$ it follows from (\ref{eq:moments}) by
induction that $\ev(Y^p)$ is a rational function of the moments
$\mu_1, \dots ,$ $\mu_p$:
\begin{eqnarray} \label{eq:fraction}
\ev(Y^p) = \frac{1}{(1-\mu_1) \cdots (1-\mu_p)} \sum_{(j_1, \dots ,
j_{p-1}) \in \{0, 1\}^{p-1}} a_{j_1, \dots , j_{p-1}}^{(p)}
\mu_1^{j_1} \cdots \mu_{p-1}^{j_{p-1}} ,
\end{eqnarray}
where the coefficients of the numerator are \emph{universal} constants,
independent of the distribution of $\xi_j$.

These universal coefficients $a_{j_1, \dots , j_{p-1}}^{(p)}$ in the
numerator of (\ref{eq:fraction}) make a symmetrical, Pascal's
triangle-like table if they are ordered according to multiindices
$(j_{p-1}, \dots , j_1)$ as binary numbers, see the rows $p = 1,
\dots , 5$:
\[
\begin{array}{ccccccccccccccccc}
&&&&&&&&1&&&&&&&& \\
&&&&&&&1&&1&&&&&&& \\
&&&&&&1&2&&2&1&&&&&& \\
&&&&1&3&5&3&&3&5&3&1&&&& \\
1&4&9&6&9&16&11&4&&4&11&16&9&6&9&4&1 \\
\end{array}
\]
Rather interestingly, based on the above table, one may conjecture
that each coefficient $a_{j_1, \dots , j_{p-1}}^{(p)}$ is equal to
the number of permutations $\pi$ of the set $\{1, \dots , p\}$
which have descent $\pi(i) > \pi(i+1)$ exactly where $j_i=1$, $1
\le i \le p-1$. Another interesting conjecture is that the
following recursion holds:
\[
a_{i_1, \dots , i_p}^{(p+1)} = \sum_{(j_1, \dots , j_{p-1})
\in S(i_1, \dots , i_p)} a_{j_1, \dots , j_{p-1}}^{(p)},
\qquad (i_1, \dots , i_p) \in \{0, 1\}^p, \qquad p \ge 1,
\]
where $S(i_1, \dots , i_p)$ is the set of all distinct binary
sequences obtained from $(i_1, \dots , i_p)$ by deleting exactly
one digit; $a^{(1)} = 1$. For example, $a_{0110}^{(5)} = 11 =
a_{110}^{(4)} + a_{010}^{(4)} + a_{011}^{(4)}$. This recursion
would imply that the above table contains only positive integers
and has the symmetries $a_{i_1, \dots , i_p}^{(p+1)} = a_{i_p,
\dots , i_1}^{(p+1)} = a_{1-i_1, \dots , 1-i_p}^{(p+1)} =
a_{1-i_p, \dots , 1-i_1}^{(p+1)}$.

There is a nice analogy between the moments of the exponential
functional of a subordinator and the moments of $Y$, compare
(\ref{eq:Jp}) and (\ref{eq:fraction}). First, the sum of the
coefficients in the numerator of (\ref{eq:fraction}) is $p!$, as
can be seen by induction. For, if one explicitly writes down
$\ev(Y^p)$, based on the recursion (\ref{eq:moments}), taking a
common denominator, the numerator of each earlier term except the
last one is multiplied by factors $1 - \mu_k$. In the sum of the
coefficients of the numerator it means a multiplication by zero.
On the other hand, in the last term one multiplies the numerator
of $\ev(Y^{p-1})$ by $p \mu_{p-1}$, which results the sum $p!$ of
the coefficients by the induction.

Second, there is a relationship between the denominators of
(\ref{eq:Jp}) and (\ref{eq:fraction}) as well. In the special case
when $Y$ is defined as in (\ref{eq:Y}), but $S_n$ is the partial
sum of an arbitrary i.i.d. sequence $(X_j)_{j=1}^{\infty}$ with zero
expectation, $\Phi(\lambda) = -n^{-1} \log
\ev\left(\exp(\lambda(S_n - \nu n))\right)
= -\log \ev(\xi^{\lambda})$, so $\Phi(k) = -\log \mu_k$, corresponding
to the factors in the denominator of (\ref{eq:Jp}). The factors
$1 - \mu_k$ in the denominator of (\ref{eq:fraction}) are tangents
to these.

Finally, let us consider the moments of $Y$ in our \emph{basic case}.
Then $\mu_k = \ev(\xi^k) = \exp(-k\nu) \cosh(k)$. Since $\cosh(k) <
e^k$ for any $k > 0$, it follows that $\mu_k < 1$ for any $k \ge 1$
when $\nu \ge 1$, therefore all positive integer moments of $Y$ are
finite in this case by Theorem \ref{th:moments}. In particular, in
Section 2 we saw that $Y$ is a bounded random variable if $\nu > 1$,
hence the positive integer moments characterize its distribution. On
the other hand, when $0 < \nu < 1$, only finitely many moments of $Y$
are finite. For, $\mu_k \ge 1$ if $0 < \nu \le k^{-1} \log \cosh(k)
\nearrow 1$ as $k \to \infty$. For example, even $\mu_1 \ge 1$ (and
consequently all $\ev(Y^p) = \infty$) if $0 < \nu < \log \cosh(1)
\approx 0.43378$.

\section{Approximation of the exponential functional of BM}

In this final section we are going to show that taking a suitable
nested sequence of RWs the resulting sequence of discrete exponential
functionals (\ref{eq:Y}) converges almost surely to the corresponding
exponential functional $\mathcal{I}$ of BM. Based on this, using
convergence of moments, we will give an elementary proof of
theorem (\ref{eq:I}) of Dufresne and Yor.

The underlying RW construction of BM was first introduced by
\cite{Knight (1962)}, and simplified and somewhat improved by
\cite{Revesz (1990)} and \cite{Szabados (1996)}. This construction
starting from an independent sequence of RWs $(S_m(k), k \ge 0)
_{m=0}^{\infty}$, constructs a dependent sequence of RWs
$(\widetilde{S}_m(k), k \ge 0)_{m=0}^{\infty}$ by "twisting" so that
the shrunken and linearly interpolated sequence $(B_m(t)
= 2^{-m}$ $\widetilde{S}_m(t2^{2m}), t \ge 0)_{m=0}^{\infty}$ a.s.
converges to BM $\left(B(t), t \ge 0\right)$, uniformly on bounded
intervals, see Theorem 3 in \cite{Szabados (1996)}.

We need one more result about this approximation here. This is stated
in a somewhat sharper form than in the cited reference, but can be
read easily from the proof there. Namely, see Lemma 4 in
\cite{Szabados (1996)}, for almost every $\omega$ there exists an
$m_0(\omega)$ such that for any $m \ge m_0(\omega)$ and for any $K
\ge e$, one has
\begin {equation} \label{eq:Bmj}
\sup_{j \ge 1} \sup_{0 \le t \le K} |B_{m+j}(t) - B_m(t)|
\le K^{\frac14} (\log K)^{\frac34} m 2^{-\frac{m}{2}} .
\end{equation}

\begin{lem} \label{le:asconv}
Let $B_m(t) = 2^{-m} \widetilde{S}_m(t2^{2m})$, $t \ge 0$, $m \ge 0$,
be a sequence of shrunken simple symmetric RWs that a.s. converges to
BM $(B(t), t \ge 0)$, uniformly on bounded intervals. Then for any
$\nu >0$, as $m \to \infty$,
\begin{eqnarray*}
&& Y_m = 2^{-2m} \sum_{k=0}^{\infty} \exp \left(2^{-m} \widetilde{S}_m(k)
- \nu k2^{-2m}\right) \\
& \to & \mathcal{I} = \int_0^\infty \exp\left(B(t)
- \nu t \right) \di t  < \infty \qquad \mbox{a.s.}
\end{eqnarray*}
\end{lem}

\begin{proof}
The basic idea of the proof is that the sequence of functions $f_m(t,
\omega) = \exp\left(B_m(t) - \nu t\right)$, converges to $f(t, \omega)
= \exp\left(B(t) - \nu t \right)$ for $t \in [0, \infty)$ as $m \to
\infty$, for almost every $\omega$. If one can find a function
$g(t, \omega) \in L^1[0, \infty)$, that dominates each $f_m$ for
$m \ge m_0(\omega)$, then their integrals on $[0, \infty)$ also
converge to the integral of $f$, and then we are practically done.

First, by (\ref{eq:Bmj}), for a.e. $\omega$ there exists an $m_0 =
m_0(\omega)$ so that for any $K \ge e$,
\begin {equation} \label{eq:Bm0}
\sup_{m \ge m_0} \sup_{0 \le t \le K} |B_m(t) - B_{m_0}(t)|
\le K^{\frac14} (\log K)^{\frac34} \le K^{\frac12} \log K ,
\end{equation}
where we supposed that $m_0$ was chosen large enough so that $m_0
2^{-m_0/2} \le 1$.

Second, by the law of iterated logarithms,
\[
\limsup_{t \to \infty} \frac{B_{m_0}(t)}{(2t \log\log t)^{\frac12}} =
\limsup_{u \to \infty} \frac{\widetilde{S}_{m_0}(u)}{(2u \log\log u)
^{\frac12}} = 1 \qquad \mbox{a.s.},
\]
where $u = t 2^{2m_0}$. Hence for a.e. $\omega$, there is a $K_0 =
K_0(\nu, \omega)$, such that for any $t \ge K_0$,
\begin{equation} \label{eq:LIL}
B_{m_0}(t) \le 2 (t \log\log t)^{\frac12} \le 2 t^{\frac12} \log t ,
\end{equation}
where $K_0$ is chosen so large that $3 t^{\frac12} \log t \le \nu
t/2$ for any $t \ge K_0$.

Since a.s. any path of $B_{m_0}$ is continuous, it is bounded on the
interval $[0, K_0]$. Then by (\ref{eq:Bm0}), we have an upper bound
uniform in $m$: for any $m \ge m_0$ and $t \in [0, K_0]$, $B_m(t) \le
M(\omega)$.  On the other hand, when $t > K_0$, by (\ref{eq:LIL}),
$B_{m_0}(t) \le 2 t^{\frac12} \log t$ and so by (\ref{eq:Bm0}),
$B_m(t) \le 3 t^{\frac12} \log t$, for any $m \ge m_0$.

Summarizing, the function
\[
g(t, \omega) = \left\{
\begin{array}{ll}
e^{M(\omega)} & \mbox{if $0 \le t \le K_0(\nu, \omega)$,} \\
e^{-\nu t/2}  & \mbox{if $t > K_0(\nu, \omega)$,}
\end{array}
\right.
\]
is an integrable function on $[0, \infty)$, dominating
$\exp(B_m(t) - \nu t)$ for each $m \ge m_0(\omega)$. This implies
that
\[
\lim_{m \to \infty} \int_0^\infty \exp\left(B_m(t) - \nu t \right)
\di t = \int_0^\infty \exp\left(B(t) - \nu t \right) \di t < \infty
\qquad \mbox{a.s.}
\]

Finally, compare $\int_0^\infty \exp\left(B_m(t) - \nu t \right)
\di t$ to $Y_m = 2^{-2m} \sum_{k=0}^{\infty} \exp \left(B_m(k 2^{-2m})
\right.$ $ \left. - \nu k2^{-2m}\right)$ that appears in the
statement of the lemma. Applying the uniform domination of
$\exp\left(B_m(t) - \nu t \right)$ by the function $g$ shown
above, both the tail of the integral on the interval $[K_0,
\infty)$ and the tail of the sum for $k \ge \lceil K_0 2^{2m}
\rceil$ is smaller than $\int^{\infty}_{K_0} \exp (-\nu t/2) \di
t$, thus their difference is uniformly arbitrarily small for any
$m \ge m_0$ if $K_0$ is large enough. On the interval $[0, K_0]$
the difference of the integral and the sum (which is a Riemann sum
of a continuous function) tends to zero uniformly as $m \to
\infty$, since on each subinterval of length $2^{-2m}$, the
difference of $B_m(t)$ and $B_m(k 2^{-2m})$ is at most $2^{-m}$.
This completes the proof of the lemma.
\end{proof}

Next we want to apply the results of the previous sections to $Y_m$.
To do this we introduce the following notations. For $m \ge 0$ and $n
\ge 1$ let
\begin{eqnarray}
Y_{m,n} &=& 2^{-2m} \sum_{k=0}^n \exp \left(2^{-m} \widetilde{S}_m(k)
- \nu k2^{-2m} \right) \nonumber \\
&=& 2^{-2m} \left(1 + \xi_{m1} + \xi_{m1} \xi_{m2} + \xi_{m1} \cdots
\xi_{mn} \right) \label{eq:Ymn}
\end{eqnarray}
and $Y_{m,0} = 2^{-2m}$, where $\xi_{mj}=\exp (2^{-m}
\widetilde{X}_m(j)-\nu 2^{-2m})$. Here
$\widetilde{X}_m(j) = \widetilde{S}_m(j) - \widetilde{S}_m(j-1)$
$(j=1,2,\dots)$ is an i.i.d. sequence, $\pr(\widetilde{X}_m(j) =
\pm 1)= \frac12$.

Then $Y_{m,n} \nearrow Y_m$ as $n \to \infty$, $Y_m < \infty$ a.s.
iff $\nu > 0$, and $Y_m$ satisfies the following self-similarity in
distribution:
\begin{equation} \label{eq:Ymself}
Y_m \law 2^{-2m} + \xi_m Y_m \qquad \mbox{or} \qquad
2^{2m} Y_m \law 1 + \xi_m 2^{2m} Y_m ,
\end{equation}
where $\xi_m$ and $Y_m$ are independent, $\xi_m \law \xi_{mj}$.
Using the notations of Section 2, now $\gamma_1 = \exp(-2^{-m} -
\nu 2^{-2m})$, $\gamma_2 = \exp(2^{-m} - \nu 2^{-2m})$, $p_1 =
p_2 = \frac12$. If $\nu > 2^m$, $\gamma_2 < 1$ holds, so the
similarity transformations $T_1$ and $T_2$ are contractions, mapping
the interval $I = [(1 - \gamma_1)^{-1}, (1 - \gamma_2)^{-1}]$
into itself. By Theorem \ref{th:singular}, the distribution of
$Y_m$ is singular w.r.t. Lebesgue measure if $\nu > 2^{2m} \log 2$
$(m \ge 1)$. Moreover, there is no overlap in the ranges of $T_1$ and
$T_2$ iff $\nu > 2^{2m} \log(2\cosh(2^{-m}))$. As $m \to \infty$
this means asymptotically that $\nu > \frac12 + 2^{2m} \log 2 +
o(1)$.

For $m \ge 0$ and $k$ integer let
\[
\mu_{mk} = \ev(\xi_m^k) = \exp(-\nu k 2^{-2m}) \cosh(k2^{-m}) .
\]
Since Theorem \ref{th:moments} is applicable to $2^{2m} Y_m$, one
obtains that $\ev(Y_m^p) < \infty$ if and only if $\mu_{mp} < 1$ and
then the following recursion is valid for $p \ge 1$ integer:
\begin{equation} \label{eq:Ymmoment}
\ev(Y_m^p) = \frac{1}{1-\mu_{mp}} \sum_{k=0}^{p-1} \binom{p}{k}
2^{-2m(p-k)} \mu_{mk} \ev(Y_m^k).
\end{equation}
Now using $\cosh(x) < e^x$ $(x > 0)$, it follows that
$\mu_{mk} < \exp\left(k2^{-m}(1 - \nu 2^{-m})\right)$. So
$\mu_{mk} < 1$ for any $k \ge 1$ if $\nu \ge 2^m$. If $0 < \nu
< 2^m$, only finitely many positive moments are finite, since
$\mu_{mk} \ge 1$ when $0 < \nu < 2^{2m} k^{-1} \log\cosh(k 2^{-m})
\to 2^m$ as $k \to \infty$.

More importantly,
\begin{equation} \label{eq:muless}
\mu_{mk} < \exp\left(k 2^{-2m} \left(\frac{k}{2} - \nu \right)
\right) \qquad (k \ge 1) ,
\end{equation}
since $\cosh(x) < \exp(x^2/2)$ when $x > 0$ (compare the Taylor
series). Thus $\mu_{mk} < 1$ for any $m \ge 0$ if $\nu \ge
\frac{k}{2}$. This condition is sharp as $m \to \infty$. For,
apply $e^x = 1 + x + o(x)$ and $\cosh(x) = 1 + x^2/2 + o(x^2)$ (as
$x \to 0$) to the definition of $\mu_{mk}$. Then
\begin{equation} \label{eq:mumk}
\mu_{mk} = 1 + k 2^{-2m} \left(\frac{k}{2} - \nu\right) + o(2^{-2m}) ,
\end{equation}
for any fixed $k$ as $m \to \infty$.

\begin{lem} \label{le:posmoment}
If $p$ is a positive integer such that $\frac{p}{2} < \nu$, then
\begin{equation} \label{eq:posmoment}
\lim_{m \to \infty} \ev(Y_m^p) = \frac{1}{\prod_{k=1}^p \left(
\nu - \frac{k}{2}\right)} < \infty .
\end{equation}
\end{lem}

\begin{proof}
By (\ref{eq:muless}), for any positive integer $p$ such that
$\frac{p}{2} < \nu$ we have $\mu_{mp} < 1$.  Since Theorem
\ref{th:moments} is valid for $2^{2m} Y_m$, the recursion formula
(\ref{eq:Ymmoment}) holds, and by induction one gets $\ev(Y_m^p)$ as a
rational function of the moments $\mu_{m1}, \dots, \mu_{mp}$,
similarly to formula (\ref{eq:fraction}). The argument below formula
(\ref{eq:fraction}) also applies here too, showing that the sum of the
coefficients in the numerator of this rational function is $p!
2^{-2mp}$. The extra factor comes from the difference that $Y_m$ is
multiplied by $2^{2m}$ here, compare equations (\ref{eq:moments}) and
(\ref{eq:Ymmoment}). Since each $\mu_{mk} \to 1 $ as $m \to \infty$,
it follows that $2^{2mp}$ times the numerator tends to $p!$.

By (\ref{eq:mumk}), we get that $1 - \mu_{mk} = k 2^{-2m}(\nu -
\frac{k}{2}) + o(2^{-2m})$ if $k$ is fixed and $m \to \infty$. So
$2^{2mp}$ times the denominator of the rational function tends to
$p! \prod_{k=1}^p (\nu - \frac{k}{2})$ as $m \to \infty$. This and
the limit of the numerator together imply the statement of the
lemma.
\end{proof}

Our next objective is to give an asymptotic formula, similar to
({\ref{eq:posmoment}), for the negative moments of $Y_m$ as $m \to
\infty$.

\begin{lem} \label{le:negmoment}
For all integer $p \ge 1$, we have
\begin{equation} \label{eq:negmoment}
\lim_{m \to \infty} \ev(Y_m^{-p}) = \lim_{m \to \infty}
\ev(Y_m^{-1}) \: \prod_{k=1}^{p-1} \left(\nu + \frac{k}{2}\right) ,
\end{equation}
where $\lim_{m \to \infty} \ev(Y_m^{-1}) < \infty$.
\end{lem}

\begin{proof}
We want to show (\ref{eq:negmoment}) by establishing a recursion.
Introduce the notations $z_{m,k} = \ev(Y_m^{-k})$ and $\mu_{m,-k} =
\ev(\xi_m^{-k})$ for $k \ge 1$ integer. By (\ref{eq:Ymn}), $0 <
Y_m^{-1} < 2^{2m}$, hence all negative moments $z_{m,k}$ of $Y_m$ are
finite.

The self-similarity equation (\ref{eq:Ymself}) implies that $\xi_m
Y_m \law Y_m - 2^{-2m}$ and so
\[
\xi_m^{-1} Y_m^{-1} \law \frac{Y_m^{-1}}{1 - 2^{-2m} Y_m^{-1}} ,
\]
where $\xi_m$ and $Y_m$ are independent. Taking $k$th moment ($k \ge
1$ integer) on both sides and applying the Taylor series
\[
\frac{x^k}{(1-x)^k} = \sum_{n=k}^\infty \binom{n-1}{k-1} x^n ,
\]
valid for any $|x| < 1$, one obtains
\[
\mu_{m,-k} \: z_{m,k} = \sum_{n=k}^\infty \binom{n-1}{k-1}
2^{-2m(n-k)}  z_{m,n} ,
\]
with the notations introduced above. This implies that
\begin{equation} \label{eq:recurs}
(\mu_{m,-k} - 1) z_{m,k} - k 2^{-2m} z_{m,k+1} = a(m,k) ,
\end{equation}
where
\[
a(m,k) =  \sum_{n=k+2}^\infty \binom{n-1}{k-1} 2^{-2m(n-k)}  z_{m,n}
\ge 0 .
\]

Next we want to give an upper bound for $a(m,k)$, which goes to zero
fast enough as $m \to \infty$. Since $\xi_m \ge \gamma_1 =
\exp(-2^{-m} - \nu 2^{-2m})$, by (\ref{eq:Ymn}) it follows that
\[
Y_m^{-1} \le 2^{2m} \left(\sum_{j=0}^{\infty} \gamma_1^j\right)^{-1}
= 2^{2m} (1 - \gamma_1) \le 2^{2m} (2^{-m} + \nu 2^{-2m})
\le 2^{m+1} ,
\]
if $m \ge \log(\nu) / \log(2)$, where we used that $1 - e^{-x} \le x$,
for any real $x$. This implies that $z_{m,r+j} = \ev(Y_m^{-r-j}) \le
\ev(Y_m^{-r}) \left( \sup(Y_m^{-1}) \right)^j \le z_{m,r} 2^{(m+1)j}$
for $r, j \ge 0$. Substituting this into the definition of $a(m,k)$,
one gets that
\begin{eqnarray*}
a(m,k) &\le& z_{m,k+1} \sum_{n=k+2}^\infty \binom{n-1}{k-1}
2^{-2m(n-k)} 2^{(m+1)(n-k-1)} \\
&=& z_{m,k+1} 2^{k(m-1)-m-1} \sum_{n=k+2}^\infty \binom{n-1}{k-1}
2^{(1-m)n} \\
&=& z_{m,k+1} 2^{-m-1} \left((1 - 2^{1-m})^{-k} - 1 - k
2^{1-m} \right) \\
&\le& z_{m,k+1} 2^{-m-1} 4k(k+1)2^{-2m} = z_{m,k+1} 2k(k+1) 2^{-3m} ,
\end{eqnarray*}
if $m$ is large enough, depending on $k$.

Let us substitute this estimate of $a(m,k)$ into (\ref{eq:recurs})
and express the following ratio:
\[
\frac{z_{m,k+1}}{z_{m,k}} =  \frac{\mu_{m,-k} - 1}{k 2^{-2m}
(1 + O(2^{-m}))}.
\]
Apply the asymptotics (\ref{eq:mumk}) here with $-k$:
\[
\frac{z_{m,k+1}}{z_{m,k}} =  \frac{\nu + \frac{k}{2} + o(1)}
{1 + O(2^{-m})} ,
\]
as $m \to \infty$. This implies the equality
$\lim_{m \to \infty} z_{m,k+1}/z_{m,k} =
\nu + \frac{k}{2}$ ,
and thus for any positive integer $p$,
\begin{equation} \label{eq:negmomrat}
\lim_{m \to \infty} \frac{\ev(Y_m^{-p})}{\ev(Y_m^{-1})}
= \prod_{k=1}^{p-1} \left(\nu + \frac{k}{2}\right) .
\end{equation}

It remains to show that $\ev(Y_m^{-1})$ has a finite limit as $m \to
\infty$. Writing $Y_m^{-2} = Y_m Y_m^{-3}$ and applying the
Cauchy-Schwarz inequality, one obtains $\left(\ev(Y_m^{-2})\right)^2
\le \ev(Y_m^2) \ev(Y_m^{-6})$, or $\ev(Y_m^{-2}) \le \ev(Y_m^2)
\ev(Y_m^{-6})/\ev(Y_m^{-2})$. Suppose first that $\nu > 1$ and take
limit here on the right hand side as $m \to \infty$, applying
(\ref{eq:posmoment}) and (\ref{eq:negmomrat}). It follows that
$\sup_{m \ge 1} \ev(Y_m^{-2}) \le \infty$. As $\ev(Y_m^{-2})$ is an
increasing function of $\nu$ by its definition, hence the same is true
for any $\nu \in (0, 1]$ as well. Since by Lemma \ref{le:asconv}, $Y_m
\to \mathcal{I}$ a.s., where each $Y_m$ and also $\mathcal{I}$ take
values in $(0, \infty)$ a.s., it follows that $Y_m^{-1} \to I^{-1}$
a.s. Then by the $L^2$ uniform boundedness of $Y_m^{-1}$ $(m \ge 0)$
shown above, $\ev(Y_m^{-1}) \to \ev(I^{-1}) < \infty$ follows as
well. This ends the proof of the lemma.
\end{proof}

Finally, it turns out that $Y_m^{-1}$ converges to $\mathcal{I}^{-1}$
in any $L^p$. This makes it possible to recover the result
(\ref{eq:I}) of \cite{Dufresne (1990)} and \cite{Yor (1992)}.

\begin{thm} \label{th:gamma}
Let $B_m(t) = 2^{-m} \widetilde{S}_m(t2^{2m})$, $t \ge 0$, $m \ge 0$,
be a sequence of shrunken simple symmetric RWs that a.s. converges to
BM $(B(t), t \ge 0)$, uniformly on bounded intervals. Take
\[
Y_m = 2^{-2m} \sum_{k=0}^{\infty} \exp \left(B_m(k 2^{-2m})
- \nu k2^{-2m}\right) = 2^{-2m} \left(1 + \xi_{m1} + \xi_{m1} \xi_{m2} +
\cdots \right)
\]
and
\[
\mathcal{I} = \int_0^\infty \exp\left(B(t) - \nu t \right) \di t
\]
when $\nu > 0$. Then the following statements hold true:

(a) $Y_m^{-1}$ converges to $\mathcal{I}^{-1}$ in $L^p$ for any
$p \ge 1$ real and $\lim_{m \to \infty} \ev(Y_m^{-p}) =
\ev(\mathcal{I}^{-p}) < \infty$;

(b) $\mathcal{I} \law 2 / Z_{2\nu}$, where $Z_{2\nu}$ is a gamma
distributed random variable with index $2\nu$ and parameter $1$;

(c) $Y_m$ converges to $\mathcal{I}$ in $L^p$ for any integer
$p$ such that $1 \le p < 2\nu$ (supposing $\nu > \frac12$) and
then $\lim_{m \to \infty} \ev(Y_m^p) = \ev(\mathcal{I}^p) <
\infty$. The same is true for any real $q$, $1 \le q < p$.
\end{thm}

\begin{proof}
By Lemma \ref{le:asconv}, $Y_m \to \mathcal{I}$ a.s., where each
$Y_m$ and also $\mathcal{I}$ take values in $(0, \infty)$ a.s. Hence
$Y_m^{-1} \to \mathcal{I}^{-1}$ a.s. By Lemma \ref{le:negmoment},
$\lim_{m \to \infty} \ev(Y_m^{-k}) < \infty$ for any $k \ge 1$
integer, so (a) follows.

Thus by (a) and Lemma \ref{le:negmoment}, for any integer $p \ge 1$,
\[
a_p = \ev(\mathcal{I}^{-p}) = c \prod_{k=1}^{p-1} \left(\nu +
\frac{k}{2}\right) = c 2^{1-p} (2 \nu + 1) \cdots (2 \nu + p -1)
< \infty ,
\]
where $c = \ev(\mathcal{I}^{-1})$. By a classical result, see
\cite{Simon (1998)}, a Stieltjes moment problem is determinate,
that is the moments uniquely determine a probability distribution on
$[0, \infty)$, if there exist constants $C>0$ and $R>0$ such that
$a_p \le C R^p (2p)!$ for any $p \ge 1$ integer. In the present case
$a_p \le c 2^{-p} (p+1)!$ when $\nu \le 1$ and $a_p \le c
(\nu / 2)^{p-1} (p+1)!$ when $\nu > 1$, so the moment problem for
$\mathcal{I}^{-1}$ is determinate and it also follows that
$\mathcal{I}^{-1}$ has a finite moment generating function in a
neighborhood of the origin.

Also, using the moments of the gamma distribution we get
\[
b_p = \ev(2^{-p} Z_{2\nu}^p) = 2^{-p}
\frac{\Gamma(2\nu + p)}{\Gamma(2\nu)} = 2^{-p} (2 \nu)(2 \nu + 1)
\cdots (2 \nu + p -1) ,
\]
for any $p \ge 1$, and $Z_{2 \nu} / 2$ has a finite moment
generating function in a neighborhood of the origin as well.
Writing down the two moment generating functions by the help of
the moments $a_p$ and $b_p$, respectively, it follows that
\[
\ev(\exp(u \mathcal{I}^{-1})) = \frac{c}{\nu} \ev(\exp(u Z_{2\nu}
/ 2))
\]
in a neighborhood of the origin. Substituting $u = 0$, one  obtains
$c = \ev(\mathcal{I}^{-1}) = \nu$ and this proves (b).

Finally, again, $Y_m \to \mathcal{I}$ a.s. by Lemma
\ref{le:asconv}. If $p$ is an integer such that $1 \le p < 2\nu$,
by Lemma \ref{le:posmoment}, using the moments of the gamma
distribution, and by (b), we have $\lim_{m \to \infty} \ev(Y_m^p)
= \ev(2^p Z_{2 \nu}^{-p}) = \ev(\mathcal {I}^{p})$. This proves
(c).
\end{proof}



\end{document}